\documentclass[11pt]{amsart}
\usepackage{amssymb,amscd,amsthm,verbatim}
\usepackage{amsbsy}
\usepackage{latexsym}
\usepackage[all,cmtip]{xy}
\usepackage[mathscr]{euscript}
\usepackage{tikz}
\usepackage{verbatim}

\usetikzlibrary{arrows}
\usetikzlibrary{decorations.markings}

\usepackage{etex}
\usepackage{graphicx}
\usepackage{pictex}
\usepackage{epstopdf}
\usepackage{color}

\usepackage{bbm}

\date{\today}

\newcommand{\Z}{{\mathbb Z}}
\newcommand{\R}{{\mathbb R}}

\newcommand{\D}{{\mathbb D}}

\newcommand{\N}{{\mathbb N}}

\newcommand{\Leb}{{\mathrm{Leb}}}

\newtheorem{theorem}{Theorem}[section]
\newtheorem{lemma}[theorem]{Lemma}
\newtheorem{prop}[theorem]{Proposition}
\newtheorem{coro}[theorem]{Corollary}

\theoremstyle{definition}

\newtheorem{remark}[theorem]{Remark}

\theoremstyle{plain}

\allowdisplaybreaks \numberwithin{equation}{section}

\begin{document}

\title[Generic Spectral Results for CMV Matrices]{Generic Spectral Results for CMV Matrices with Dynamically Defined Verblunsky Coefficients}

\author[L.\ Fang]{Licheng Fang}
\address{Ocean University of China, Qingdao 266100, Shandong, China}
\email{flcouc@gmail.com}

\thanks{L.F.\ was supported by NSFC (No. 11571327) and the Joint PhD Scholarship Program of Ocean University of China}

\author[D.\ Damanik]{David Damanik}
\address{Department of Mathematics, Rice University, Houston, TX~77005, USA}
\email{damanik@rice.edu}
\thanks{D.D.\ was supported in part by NSF grant DMS--1700131 and by an Alexander von Humboldt Foundation research award}

\author[S.\ Guo]{Shuzheng Guo}
\address{Ocean University of China, Qingdao 266100, Shandong, China and Rice University, Houston, TX~77005, USA}
\email{gszouc@gmail.com}
\thanks{S.G.\ was supported by CSC (No. 201906330008) and the Fundamental Research Funds for the Central Universities (N0. 201861004)}

\begin{abstract}
We consider CMV matrices with dynamically defined Verblunsky coefficients. These coefficients are obtained by continuous sampling along the orbits of an ergodic transformation. We investigate whether certain spectral phenomena are generic in the sense that for a fixed base transformation, the set of continuous sampling functions for which the spectral phenomenon occurs is residual. Among the phenomena we discuss are the absence of absolutely continuous spectrum and the vanishing of the Lebesgue measure of the spectrum.
\end{abstract}

\maketitle

\section{Introduction}

This paper is concerned with the spectral analysis of (extended) CMV matrices with dynamically defined Verblunsky coefficients. CMV matrices are canonical matrix representations of unitary operators with a cyclic vector, and they arise naturally in the context of orthogonal polynomials on the unit circle. We refer the reader to \cite{S04, S05} for background.

Let us recall how CMV matrices arise in connection with orthogonal polynomials on the unit circle. Suppose $\mu$ is a non-trivial probability measure on the unit circle $\partial \mathbb{D} = \{ z \in \mathbb{C} : |z| = 1 \}$, that is, $\mu(\partial \D) = 1$ and $\mu$ is not supported on a finite set. By the non-triviality assumption, the functions 1,$z$, $z^2,\cdots$ are linearly independent in the Hilbert space $\mathcal{H} = L^2(\partial\mathbb{D}, d\mu)$, and hence one can form, by the Gram-Schmidt procedure, the \emph{monic orthogonal polynomials} $\Phi_n(z)$, whose Szeg\H{o} dual is defined by $\Phi_n^{*} = z^n\overline{\Phi_n({1}/{\overline{z}})}$. There are constants $\{\alpha_n\}_{n=0}^\infty$ in $\mathbb{D}=\{z\in\mathbb{C}:|z|<1\}$, called the \emph{Verblunsky coefficients}, so that
\begin{equation}\label{eq1}
\Phi_{n+1}(z) = z \Phi_n(z) - \overline{\alpha}_n \Phi_n^*(z),
\end{equation}
which is the so-called \emph{Szeg\H{o} recurrence}. Conversely, every sequence $\{\alpha_n\}_{n=0}^\infty$ in $\mathbb{D}$ arises in this way.

The orthogonal polynomials may or may not form a basis of $\mathcal{H}$. However, if we apply the Gram-Schmidt procedure to $1, z, z^{-1}, z^2, z^{-2}, \ldots$, we will obtain a basis -- called the \emph{CMV basis}. In this basis, multiplication by the independent variable $z$ in $\mathcal{H}$ has the matrix representation
\begin{equation*}
\mathcal{C}=\left(
\begin{matrix}
\overline{\alpha}_0&\overline{\alpha}_1\rho_{0}&\rho_1\rho_0&0&0&\cdots&\\
\rho_0&-\overline{\alpha}_1\alpha_{0}&-\rho_1\alpha_0&0&0&\cdots&\\
0&\overline{\alpha}_2\rho_{1}&-\overline{\alpha}_2\alpha_{1}&\overline{\alpha}_3\rho_2&\rho_3\rho_2&\cdots&\\
0&\rho_2\rho_{1}&-\rho_2\alpha_{1}&-\overline{\alpha}_3\alpha_2&-\rho_3\alpha_2&\cdots&\\
0&0&0&\overline{\alpha}_4\rho_3&-\overline{\alpha}_4\alpha_3&\cdots&\\
\cdots&\cdots&\cdots&\cdots&\cdots&\cdots&
\end{matrix}
\right),
\end{equation*}
where $\rho_n := (1-|\alpha_n|^2)^{1/2}$ for $n \ge 0$. A matrix of this form is called a \emph{CMV matrix}.

It is sometimes helpful to also consider a two-sided extension of a matrix of this form. Namely, given a bi-infinite sequence $\{ \alpha_n \}_{n \in \Z}$ in $\D$ (and defining the $\rho_n$'s as before), we may consider the \emph{extended CMV matrix}
\begin{equation*}
\mathcal{E}=\left(
\begin{matrix}
\cdots&\cdots&\cdots&\cdots&\cdots&\cdots&\cdots\\
\cdots&-\overline{\alpha}_0\alpha_{-1}&\overline{\alpha}_1\rho_{0}&\rho_1\rho_0&0&0&\cdots&\\
\cdots&-\rho_0\alpha_{-1}&-\overline{\alpha}_1\alpha_{0}&-\rho_1\alpha_0&0&0&\cdots&\\
\cdots&0&\overline{\alpha}_2\rho_{1}&-\overline{\alpha}_2\alpha_{1}&\overline{\alpha}_3\rho_2&\rho_3\rho_2&\cdots&\\
\cdots&0&\rho_2\rho_{1}&-\rho_2\alpha_{1}&-\overline{\alpha}_3\alpha_2&-\rho_3\alpha_2&\cdots&\\
\cdots&0&0&0&\overline{\alpha}_4\rho_3&-\overline{\alpha}_4\alpha_3&\cdots&\\
\cdots&\cdots&\cdots&\cdots&\cdots&\cdots&\cdots
\end{matrix}\right).
\end{equation*}

In this paper, we consider dynamically defined Verblunsky coefficients and consider the direct spectral problem, where we analyze the spectral properties of the associated CMV matrices. To fix such a setting, we need a \emph{base transformation} $(\Omega, T)$ and a \emph{sampling function} $f : \Omega \to \D$.

Here, $\Omega$ will be a compact metric space and $T : \Omega \to \Omega$ a homeomorphism. The function $f : \Omega \to \D$ is assumed to be continuous. The associated Verblunsky coefficients are then given by
$$
\alpha_n = \alpha_n(\omega) := f(T^n \omega).
$$
Here, $\omega \in \Omega$ is an arbitrary initial point whose orbit under $T$ is sampled using $f$. Since $T$ is invertible, we can take all $n \in \Z$ in this definition. Thus, we can consider both standard and extended CMV matrices. The connection between the two is discussed in detail in \cite{S04, S05}. It is often natural to initially study the two-sided case, even if one ultimately is interested in the one-sided case. Some tools are naturally two-sided; for example, Kotani theory. One can then pass from results for the extended matrices to results for standard matrices by restriction. The latter procedure is well-known and hence we will focus on the two-sided case in this paper.

Our goal is to establish the genericity of certain spectral phenomena. To do so, we will fix the base transformation and then ask for how many $f \in C(\Omega,\D)$ the phenomenon in question occurs. If this set of $f$'s is \emph{residual} (i.e., it contains a dense $G_\delta$ set), then we say that the phenomenon is \emph{generic}. Note that even after fixing $(\Omega, T, f)$, the operator $\mathcal{E} = \mathcal{E}_\omega$ still depends on $\omega \in \Omega$. This is dealt with by considering the concepts of minimality or ergodicity.

If $T$ is \emph{minimal} (i.e., every $T$-orbit $\{ T^n \omega : n \in \Z \}$ is dense in $\Omega$), then some spectral properties of $\mathcal{E}_\omega$ are independent of the choice of $\omega$. For example, the spectrum $\sigma(\mathcal{E}_\omega)$ and the absolutely continuous spectrum $\sigma_\mathrm{ac}(\mathcal{E}_\omega)$ are independent of $\omega$ in this case:

\begin{theorem}\label{t.main0a}
Let $(\Omega, T, f)$ be as above and suppose that $T$ is minimal. Then, there are compact sets $\Sigma, \Sigma_\mathrm{ac} \subseteq \partial \D$, such that for every $\omega \in \Omega$, we have $\sigma(\mathcal{E}_\omega) = \Sigma$ and $\sigma_\mathrm{ac}(\mathcal{E}_\omega) = \Sigma_\mathrm{ac}$.
\end{theorem}

On the other hand, there are always $T$-ergodic measures $\beta$ (i.e., $T$-invariant measures so that any invariant measurable set must have measure zero or one). For any fixed such measure, the spectrum and the spectral type of $\mathcal{E}_\omega$ will be almost surely independent of $\omega$.

\begin{theorem}\label{t.main0b}
Let $(\Omega, T, f)$ be as above and suppose that $\beta$ is a $T$-ergodic Borel probability measure. Then, there are compact sets $\Sigma, \Sigma_\mathrm{ac}, \Sigma_\mathrm{sc}, \Sigma_\mathrm{pp} \subseteq \partial \D$, such that for $\beta$-almost every $\omega \in \Omega$, we have $\sigma(\mathcal{E}_\omega) = \Sigma$ and $\sigma_\bullet(\mathcal{E}_\omega) = \Sigma_\bullet$ for $\bullet \in \{ \mathrm{ac}, \mathrm{sc}, \mathrm{pp} \}$.
\end{theorem}

These results are standard in the discrete Schr\"odinger operator setting \cite{CL90, CFKS87, D17, DF, P80}, and the statement in Theorem~\ref{t.main0b} about the spectrum being almost everywhere constant is \cite[Theorem~10.16.1]{S05}. Moreover, while Theorem~\ref{t.main0a} is not stated explicitly in \cite{S05}, it will quickly follow from the proof of \cite[Theorem~10.9.11]{S05}. The other statements in Theorems~\ref{t.main0a} and \ref{t.main0b} are currently not yet in the literature, but their proofs are very similar to the proofs of the corresponding results in the Schr\"odinger operator setting, which can be found, for example, in \cite{CL90, CFKS87, DF}. In fact, closely related work was done by Geronimo-Teplyaev in \cite{GT94}; see especially \cite[Theorem~3.3 {\&} Theorem~3.4]{GT94}. However, since \cite{GT94} was written before extended CMV matrices were introduced, that paper does not contain Theorems~\ref{t.main0a} and \ref{t.main0b} in the formulation above. Since these results are so fundamental, and yet never explicitly stated and discussed, but essentially known to experts, we try to strike a balance here by stating them explicitly as theorems and sketching their proofs in an appendix to this paper. Full details are expected to be included in the second edition of \cite{S05}.

When we want to emphasize the dependence of these sets on the sampling function, we will write $\Sigma(f), \Sigma_\mathrm{ac}(f), \Sigma_\mathrm{sc}(f), \Sigma_\mathrm{pp}(f)$. (They also depend on $(\Omega, T, \beta)$, but in what follows, the base dynamics will be fixed and hence left implicit in the notation.)

Our first generic result concerns the absolutely continuous spectrum.

\begin{theorem}\label{t.main1}
Suppose that $\Omega$ is a compact metric space, $T : \Omega \to \Omega$ is a homeomorphism, and $\beta$ is a non-atomic $T$-ergodic Borel probability measure on $\Omega$. Then,
$$
\{ f \in C(\Omega,\D) : \Sigma_\mathrm{ac}(f) = \emptyset \}
$$
is residual.
\end{theorem}

Many modern results for OPUC and CMV matrices are analogs of results first proved in the setting of OPRL and Jacobi matrices, and most often in the special case of discrete Schr\"odinger operators. Indeed, almost the entire text \cite{S05} was written in this spirit. The results presented in the present paper are of this type as well. The discrete Schr\"odinger version of Theorem~\ref{t.main1} was obtained by Avila-Damanik in \cite{AD05}.

Our second generic result concerns the Lebesgue measure of the spectrum. Here, the base dynamics will be of a particular form: we will consider aperiodic subshifts that satisfy the Boshernitzan condition. Recall that a \emph{subshift} is a closed shift-invariant subset $\Omega$ of $A^\Z$, where $A$ is a finite set carrying the discrete topology and $A^\Z$ is endowed with the product topology. The map $T : \Omega \to \Omega$ is given by the shift $(T \omega)_n = \omega_{n+1}$, and it is clearly a homeo\-morphism. We say that a subshift $\Omega$ satisfies the \emph{Boshernitzan condition} (B) if it is minimal and there is a $T$-invariant Borel probability measure $\mu$ such that
$$
\limsup_{n \to \infty} n \cdot \min \{ \mu([w]) : w \in \Omega_n \} > 0.
$$
Here $\Omega_n = \{ \omega_1 \ldots \omega_n : \omega \in \Omega \}$ is the set of words of length $n$ that occur in elements of $\Omega$ and $[w]$ is the cylinder set $[w] = \{ \omega \in \Omega : \omega_1 \ldots \omega_n = w \}$. This condition was introduced by Boshernitzan in \cite{B92} as a sufficient condition for unique ergodicity.

\begin{theorem}\label{t.main2}
Suppose that $(\Omega,T)$ is an aperiodic subshift that satisfies {\rm (B)}. Then,
$$
\{ f \in C(\Omega,\D) : \Leb(\Sigma(f)) = 0 \}
$$
is residual.
\end{theorem}

Of course zero-measure spectrum implies empty absolutely continuous spectrum, but the latter property holds under much weaker assumptions, and the former property is not expected to hold in similar generality. Theorem~\ref{t.main2} was obtained for discrete Schr\"odinger operators by Damanik-Lenz in \cite{DL19}.

\begin{remark}
There is a third generic result in the discrete Schr\"odinger operator literature. Namely, Boshernitzan-Damanik proved in \cite{BD08} that the metric repetition property implies that for a generic $f$, we have $\Sigma_\mathrm{pp}(f) = \emptyset$. Combining this with \cite{AD05} and intersecting two residual sets, we see that the almost sure presence of purely singular continuous spectrum is generic for suitable base dynamics. Boshernitzan-Damanik established the metric repetition property for several examples \cite{BD08, BD09}, including shifts and skew-shifts on tori. The reason why we do not work out a CMV version of these results is that this work has already been done by Ong \cite{O13}.
\end{remark}

\begin{remark}
As indicated above, while the generic results in Theorems~\ref{t.main1} and \ref{t.main2} are formulated for extended CMV matrices, they imply the corresponding results for standard (one-sided) CMV matrices by restriction. Specifically, for a suitable rank-one perturbation, each extended CMV matrix $\mathcal{E}_\omega$ decouples into a direct sum, one of whose summands is $\mathcal{C}_\omega$; see the second proof of Theorem~10.16.3 on p.704 of \cite{S05}. In particular, if $\mathcal{E}_\omega$ has purely singular spectrum, then $\mathcal{C}_\omega$ has purely singular spectrum, and if $\mathcal{E}_\omega$ has zero-measure spectrum, then $\mathcal{C}_\omega$ has zero-measure spectrum. Here, $\mathcal{E}_\omega$ is the extended CMV matrix with Verblunsky coefficients $\{ \alpha_n(\omega) \}_{n \in \Z}$, while $\mathcal{C}_\omega$ is the standard CMV matrix with Verblunsky coefficients $\{ \alpha_n(\omega) \}_{n \ge 0}$.
\end{remark}

The organization of this paper is as follows. In Section~\ref{s.2} we recall some key concepts from the theory of CMV matrices.  Theorems~\ref{t.main1} and \ref{t.main2} are proved in Sections~\ref{s.3} and \ref{s.4}, respectively. The appendix contains material that is crucial to our work in the main part of the paper, and which is not yet available in the literature in the form needed. Specifically, in Appendix~\ref{s.a} we sketch the proofs of Theorems~\ref{t.main0a} and \ref{t.main0b} and in Appendix~\ref{a.appb} we state (and derive from known results) the key elements of Kotani theory for extended CMV matrices.

\section*{Acknowledgment}

We are indebted to Jake Fillman for pointing out that the ideas in \cite{DFL17} should make the short proof of Theorem~\ref{t.main2} we give in this paper possible and to Fritz Gesztesy for very helpful pointers to the literature.

\section{Preliminaries}\label{s.2}

In this section we recall and collect a few tools that we will need in the proofs of Theorems~\ref{t.main1} and \ref{t.main2}, which are given in Sections~\ref{s.3} and \ref{s.4}, respectively. All these concepts are discussed in detail in \cite{S04, S05}, and we refer the reader to those monographs for more information.

As mentioned in the introduction, we are primarily interested in Verblunsky coefficients of the form $\alpha_n = \alpha_n(\omega) := f(T^n \omega)$, $\omega \in \Omega$, $n \in \Z$, where $\Omega$ is a compact metric space, $T : \Omega \to \Omega$ is a homeomorphism, and $f : \Omega \to \D$ is continuous. However, since we will want to approximate continuous $f$'s by discontinuous functions, we have to consider a more general setting, where $f$ is assumed to be measurable (and is subject to additional conditions, which will be imposed as needed).

If we fix a $T$-ergodic probability measure $\beta$, then by Theorem~\ref{t.main0b} the associated extended CMV matrices almost surely have a common spectrum $\Sigma$ in the sense that $\sigma(\mathcal{E}_\omega) = \Sigma$ for $\beta$-almost every $\omega \in \Omega$.

If we normalize the monic orthogonal polynomials $\Phi_n(z)$ by
$$
\phi_n(z)=\frac{\Phi_n(z)}{\|\Phi_n(z)\|_\mu},
$$
it is easy to see that \eqref{eq1} is equivalent to
\begin{equation}\label{eq2}
\rho_n(\omega) \phi_{n+1}(z) = z \phi_n(z) - \overline{\alpha_n(\omega)} \phi_n^*(z).
\end{equation}

Applying $^*$ to \eqref{eq2}, we thus obtain
\begin{equation}\label{eq3}
\left(
\begin{matrix}
\phi_{n}\\
\phi_{n}^*
\end{matrix}
\right)
=
T_n^z(\omega)
\left(
\begin{matrix}
\phi_0\\
\phi_0^*
\end{matrix}
\right),
\end{equation}
where
\begin{equation*}
T_n^z(\omega) = \prod_{j=n-1}^0 A^z(T^j \omega),
\end{equation*}
and
$$
A^z(\omega) = \frac{1}{\rho_0(\omega)}
\left(
\begin{matrix}
z&-\overline{\alpha_0(\omega)}&\\
-\alpha_0(\omega)z&1&
\end{matrix}
\right).
$$
When
\begin{equation}\label{e.znonzero}
z \not= 0,
\end{equation}
we can normalize $A^z(\omega)$ and get the determinant $1$ matrix $M^z(\omega)$, known as the $\mathbf{SU}(1,1)$-valued \emph{Szeg\H{o} cocycle map}, given by
\begin{equation*}
M^z(\omega)=\frac{1}{\rho_0(\omega)}
\left(
\begin{matrix}
\sqrt{z}&-\frac{\overline{\alpha_0(\omega)}}{\sqrt{z}}&\\
-\alpha_0(\omega)\sqrt{z}&\frac{1}{\sqrt{z}}&
\end{matrix}
\right).
\end{equation*}
The spectral properties of CMV matrices with dynamically defined Verblunsky coefficients may be investigated using the \emph{Szeg\H{o} cocycle} $(\omega,n) \rightarrow M_n^z(\omega)$ over $T$ given by
\begin{equation*}
M_n^z(\omega)=
\left\{
\begin{aligned}
M^z(T^{n-1}\omega)\cdots M^z (\omega),&\qquad & n> 0, \\
Id,&\qquad& n=0,\\
M^z(T^{n}\omega)^{-1}\cdots M^z (T^{-1}\omega)^{-1}, & \qquad & n<0,
\end{aligned}
\right.
\end{equation*}
where $Id$ is the $2\times 2$ identity matrix.

The \emph{Lyapunov exponent} of this cocycle is defined by
$$
\gamma_f(z) := \lim_{n \rightarrow \infty} \frac{1}{n} \int_{\Omega} \log \| M_n^z(\omega) \| \, d\beta(\omega) \geq0.
$$
By Kingman's Subadditive Ergodic Theorem, one has
$$
\gamma_f(z) = \lim_{n \rightarrow \infty} \frac{1}{n} \log \| M_n^z(\omega) \|
$$
for $\beta$-almost every $\omega$. %We say that the function $M^z$ is \textit{uniform} if the limit exists for all $\omega \in \Omega$ and the convergence is uniform.

Kotani theory provides a description of $\Sigma_\mathrm{ac}$ in terms of the Lyapunov exponent:
\begin{equation}\label{e.kotanithm}
\Sigma_\mathrm{ac} = \overline{\{ z \in \partial \D : \gamma_f(z) = 0 \}}^{\mathrm{ess}}.
\end{equation}
This result is not explicitly stated in \cite{S04} since the discussion there focuses on CMV matrices, while our discussion focuses on extended CMV matrices. It does however follow from the treatment there; compare Theorem~\ref{th1} in Appendix~B.

The two conditions under which Kotani theory is developed in \cite[Section~10.11]{S04} are
\begin{equation}\label{eq4}
\int_{\Omega} - \log (1 - |f(\omega)|) \, d\beta(\omega) < \infty
\end{equation}
and
\begin{equation}\label{eq5}
\int_{\Omega} -\log (|f(\omega)|) \, d\beta(\omega) < \infty.
\end{equation}
We remark that condition \eqref{eq4} is of critical importance to the existence of the Lyapunov exponent (which in turn is the central object in Kotani theory). On the other hand, condition \eqref{eq5} is used in \cite{S04} merely to establish the existence of the density of zeros measure as a measure on the unit circle, and only the consequence is needed, while the sufficient condition is conjecturally much too strong and should be replaced by a simple non-triviality assumption (i.e., that the Verblunsky coefficients do not vanish identically); compare the second remark after the statement of \cite[Theorem~10.5.26]{S04} and \cite[Conjecture~10.5.23]{S04}. Since for our applications of Kotani theory we do not want to have to assume condition \eqref{eq5}, we describe a version of Kotani theory for extended CMV matrices, only assuming \eqref{eq4} but not \eqref{eq5}, in Appendix~B.

For $\omega \in \Omega$, let $d\mu_\omega$ be the probability measure on $\partial\mathbb{D}$ associated with the one-sided sequence $\{ \alpha_n(\omega) \}_{n \ge 0}$. There is an associated Carath\'eodory function, $G : \mathbb{D} \rightarrow -i \mathbb{C}_+$, given by
\begin{equation*}
G(z) = \int \frac{e^{i \theta} + z}{e^{i \theta} - z} \, d\mu_\omega(\theta),
\end{equation*}
and a Schur function, $g : \mathbb{D} \rightarrow \mathbb{D}$, given by
\begin{equation*}
G(z) = \frac{1 + z g(z)}{1 - zg(z)}.
\end{equation*}

For $\lambda = e^{i \psi} \in \partial\mathbb{D}$, we define the OPUC $\phi_n^\lambda$ associated with the boundary condition $\lambda$ as follows:
$$
\phi_n^\lambda (z, d\mu_\omega) = \phi_n (z, d\mu^\lambda_\omega),
$$
where $d\mu_\omega^\lambda$, the Aleksandrov measure, is defined by
$$
\alpha_n(d\mu_\omega^\lambda) = \lambda \alpha_n(d\mu_\omega).
$$
The OPUC corresponding to the special case $\lambda=-1$ are singled out:
$$
\psi_n(z,d\mu_\omega) := \phi_n^{\lambda = -1}(z,d\mu_\omega).
$$

Define
\begin{equation*}
u_n =\psi_n + G(z) \phi_n,\  u_n^* = -\psi_n^* + G(z) \phi_n^*.
\end{equation*}
Then, \cite[Theorem 3.2.11]{S04} shows that
$\left(
\begin{matrix}
u_n\\u_n^*
\end{matrix}
\right)$
is the unique solution of
$$
\Xi_n = T_n^z(\omega) \Xi_{0}
$$
with the initial condition
$$
\Xi_{0} = \left(
\begin{matrix}
1\\
1
\end{matrix}
\right).
$$
%and the solution is unique in $\ell^2(\mathbb{Z}_+,\mathbb{C}^2)$.
Define
$$
m^+_{\omega,f}(z) = \frac{u_1(z)}{u_0(z)}.
$$
Then, \cite[Theorem~10.11.6]{S04} shows that for fixed $z \in \mathbb{D} \backslash \{ 0 \}$, $\log |m_{\omega,f}^+(z)|$ is in $L^1(\Omega, d\beta)$ and
\begin{equation}\label{e.mandgamma}
\int_\Omega \log|m_{\omega,f}^+(z)| \, d\beta(\omega) = \log|z| - \gamma_f(z).
\end{equation}

%Since $\Omega$ is compact and $f : \Omega \to \D$ is continuous, we observe that

%And $d\mu_\omega = w_\omega(\theta)\frac{d\theta}{2\pi}+d\mu_{\omega,s}$ is the corresponding measure. Theorem 10.11.6 in \cite{S04} shows that
%\[m_{\omega,f}^+(z)=\rho_0^{-1}z(1-\overline{\alpha}_0g)\]
%and for fixed $z\in\mathbb{D}\backslash\{0\}$, $\log|m_{\omega,f}^+(z)|$ is in $L^1(\Omega,d\beta)$ and
%$$
%\int_\Omega \log|m_{\omega,f}^+(z)| \, d\beta(\omega) = \log|z| - \gamma_f(z).
%$$

\section{Generic Absence of Absolutely Continuous Spectrum}\label{s.3}

In this section we discuss the generic absence of absolutely continuous spectrum and seek to prove Theorem~\ref{t.main1}. In fact, we will prove a more general result, Theorem~\ref{t.main3} below, that also allows us to include a coupling constant. This has the added benefit that one realizes in this way that, for a generic choice of a sampling function $f$, one cannot force the existence of absolutely continuous spectrum even by multiplying the sequence of Verblunsky coefficients by an arbitrarily small positive constant.

For any arc $I \subseteq \partial\mathbb{D}$, we define
$$
M(f,I) := \Leb ( \{ z \in I : \gamma_f(z) = 0 \}).
$$

\begin{remark}\label{r.absenceofaccrit}
By the key result \eqref{e.kotanithm} from Kotani theory, it follows that $I$ almost surely contains no absolutely continuous spectrum if and only if $M(f,I) = 0$.
\end{remark}

\begin{lemma}{\label{AD05L1}}
Suppose $I \subset \partial \D$ is an arc of length $|I| \le \frac{\pi}{2}$. For every $0 < r < 1$, the maps
\begin{equation}{\label{AD05(1)}}
(L^1(\Omega) \cap B_r(L^\infty(\Omega)), \|\cdot\|_1)\rightarrow \mathbb{R},\qquad f \rightarrow M(f,I),
\end{equation}
and
\begin{equation}{\label{AD05(2)}}
(L^1(\Omega) \cap B_r(L^\infty(\Omega)),\|\cdot\|_1)\rightarrow \mathbb{R},\qquad f \rightarrow \int_0^1 M(\lambda f,I) \, d\lambda,
\end{equation}
are upper semi-continuous.
\end{lemma}

\begin{remark}\label{r.nozero}
This lemma is crucial to the proof of Theorems~\ref{t.main1} and \ref{t.main3}. Its proof is based on the mean-value property of harmonic functions. Recall that the assumption \eqref{t.main0b} was necessary for our dynamical setup, compare also \eqref{e.mandgamma}. The exclusion of zero, where one would normally center the application of the mean-value property, leads to some technical difficulties that need to be overcome in the proof of the lemma (as well as its formulation).
\end{remark}

\begin{proof}
The statement of the lemma is trivial if $|I| = 0$, so let us assume that $0 < |I| \le \frac{\pi}{2}$.

We begin by remarking that we work in the $L^\infty$ ball
$$
B_r(L^\infty(\Omega)) = \{ f : \Omega \to \D : \|f\|_\infty \le r \}
$$
for the fixed chosen value of $r \in (0,1)$, hence the condition \eqref{eq4} will always be satisfied, and the existence of the Lyapunov exponent is ensured in the settings of \eqref{AD05(1)} and \eqref{AD05(2)}. Other than that, all other statements and arguments are relative to the $L^1$ norm.

It suffices to show that \eqref{AD05(1)} is upper semi-continuous since this implies that \eqref{AD05(2)} is also upper semi-continuous by Fatou's lemma.

We have to show that if $f_n, f \in L^1(\Omega) \cap B_r(L^\infty(\Omega))$ and $f_n \rightarrow f$ with respect to the $L^1$ norm, then $\limsup M(f_n,I) \leq M(f,I)$.

Assume otherwise. Then there are $f_n, f \in L^1(\Omega) \cap B_r(L^\infty(\Omega))$ such that
\begin{enumerate}

\item $f_n \rightarrow f$ in $L^1$ and pointwise as $n\rightarrow\infty$,

\item $\liminf M(f_n,I) \geq M(f,I) + \varepsilon$ for some $\varepsilon > 0$.

\end{enumerate}

By (1), we have pointwise convergence of the $m^+$ functions $m^+_{\omega,f_n}$ in $\mathbb{D}$ for almost every $\omega \in \Omega$.

In view of the relationship \eqref{e.mandgamma} between $m^+_{\omega,f}$ and $\gamma_f(z)$, that is,
$$
\int_\Omega \log|m^+_{\omega,f}(z)| \, d\beta(\omega) = \log|z| - \gamma_f(z),
$$
the associated Lyapunov exponents $\gamma_{f_n}(z)$ convergence pointwise to $\gamma_f(z)$ in $\mathbb{D} \backslash \{ 0 \}$.

\medskip

We will consider the composition of two conformal maps. The variables will be called $\zeta,w,z$, with $z$ being the spectral parameter above, and the maps will be called $w = \Phi_1(\zeta)$ and $z = \Phi_2(w)$. The goal will be to avoid $0$ in the $z$-plane (cf.~Remark~\ref{r.nozero}), but to apply the mean-value property in the $\zeta$-plane centered at $0$.

\medskip

Consider, in the $w$-plane, the region $U_1$ in the upper half plane that is bounded by the equilateral triangle $T$ with vertices $A_1$, $A_2$, $A_3$ satisfying the following properties: the common side length $\ell$ is no more than $1$ (and it will be determined subject to this condition momentarily), and $A_1A_2$  is the base of the triangle and it has $(0,0)$ as its midpoint. For example, choosing $\ell$ maximal, $\ell = 1$, gives rise to the vertices $A_1 = -\frac{1}{2}$, $A_2 = \frac{1}{2}$, and $A_3 = \frac{\sqrt{3}}{2} i$ of $T$.

Consider the conformal mapping $\Phi_1$ from the unit disk $\D$ to $U_1$. By the Schwarz-Christoffel formula (see, e.g., \cite[Theorem 2.2]{DT02}),
\begin{equation}{\label{SCformula}}
\Phi_1(\zeta) = C_1 + C_2 \int^\zeta \prod_{k=1}^{3} \left( 1 - \frac{\xi}{\zeta_k} \right)^{-\frac{2}{3}} \, d\xi,
\end{equation}
where $C_1$ and $C_2$ are constants and $\zeta_1$, $\zeta_2$, $\zeta_3$ are the inverse images under $\Phi_1$ of the vertices $A_1, A_2, A_3$ of $T$.

Consider also the conformal mapping
\begin{equation}\label{e.Phi2formula}
\Phi_2(w) = e^{i \theta^{\prime}} \frac{w-i}{w+i},
\end{equation}
from $U_1$ to $U_2 := \Phi_2(U_1)$ with the parameters $\ell$ and $\theta'$ chosen so that
\begin{equation}\label{e.fixI}
\Phi_2 (A_1 A_2) = I.
\end{equation}
In order to see that these choices are possible, choose $\theta'$ so that $e^{-i \theta^{\prime}} I$ has $-1$ as its midpoint and then note that choosing the maximal value $1$ for $\ell$ would send the base of $T$ (the line segment connecting $A_1 = -\frac{1}{2}$ and $A_2 = \frac{1}{2}$) to the arc in $\partial \D$ with midpoint $-1$ and endpoints $- \frac35 \pm \frac45 i$. Since this arc has length more than $0.58 \pi$, and $I$ has length no more than $\frac{\pi}{2}$, we can obtain the desired $e^{-i \theta^{\prime}} I$ by lowering the value of $\ell$ suitably.

Denote the image of the side $A_2 A_3$ of $T$ under $\Phi_2$ by $J$ and the image of the side $A_3 A_1$ of $T$ under $\Phi_2$ by $K$. Thus the region $U_2$ is bounded by $T^{\prime}$ with sides $I$, $J$, $K$.

Let us emphasize an important point. Since $\ell \le 1$, we have that $i$ lies outside of $\overline{U_1}$, and consequently $0$ lies outside $\overline{U_2}$. This property is crucial in light of Remark~\ref{r.nozero}.

Composing the conformal maps $\Phi_1$ and $\Phi_2$, we have
$$
\Phi_2(\Phi_1(\zeta)) =: \Phi(\zeta) : \D \rightarrow U_2.
$$

The functions $\gamma_{f_n} \circ \Phi$ are harmonic and bounded in $\mathbb{D}$ (here we use the fact that was pointed out above: $\overline{U_2}$ does not contain $0$). This yields
$$
\gamma_{f_n} (\Phi(0)) = \frac{1}{2\pi} \int_0^{2\pi} \gamma_{f_n}(\Phi(e^{i\theta})) \, d\theta,
$$
and similarly for $\gamma_f$. Since $\gamma_{f_n}(\Phi(0)) \rightarrow \gamma_{f}(\Phi(0))$ as $n \rightarrow \infty$, we infer
$$
\frac{1}{2\pi} \int_0^{2\pi} [\gamma_{f_n}(\Phi(e^{i\theta})) - \gamma_{f}(\Phi(e^{i\theta}))] \, d\theta \rightarrow 0.
$$
Since the Lyapunov exponents $\gamma_{f_n}(z)$ converge pointwise to $\gamma_f(z)$ in $\mathbb{D} \backslash \{ 0 \}$, by dominated convergence, the integral along $J$ and $K$ goes to zero. Therefore, the integral along $I$ goes to zero.

Since these remarks pertain to the $z$-plane, let us reformulate this as a statement in the $\zeta$-plane, and especially in terms of the $\theta$-variable. Choose $\theta_1,\theta_2,\theta_3 \in [0,2\pi)$ so that $\zeta_j = e^{i\theta_j}$, $j = 1,2,3$.

The observation above therefore implies that
\begin{equation}\label{e.leusc1}
\frac{1}{2\pi} \int_{\theta_1}^{\theta_2} [\gamma_{f_n}(\Phi(e^{i\theta})) - \gamma_{f}(\Phi(e^{i\theta}))] \, d\theta \rightarrow 0.
\end{equation}

By upper semi-continuity of the Lyapunov exponent (i.e., $\limsup \gamma_{f_n}(z) \leq \gamma_f(z)$ for every $z$) and dominated convergence,
\begin{equation}\label{e.leusc2}
\frac{1}{2\pi} \int_{\theta_1}^{\theta_2} \max \{ \gamma_{f_n}(\Phi(e^{i\theta})) - \gamma_{f}(\Phi(e^{i\theta})), 0 \} \, d\theta \rightarrow 0.
\end{equation}
Combining \eqref{e.leusc1} and \eqref{e.leusc2}, we find
\begin{equation}\label{e.leusc3}
\frac{1}{2\pi} \int_{\theta_1}^{\theta_2} \min \{ \gamma_{f_n}(\Phi(e^{i\theta})) - \gamma_{f}(\Phi(e^{i\theta})), 0 \} \, d\theta \rightarrow 0.
\end{equation}

Choose $\delta > 0$ such that the set $X = \{ z \in I : \gamma_f(z) < \delta \}$ has measure bounded by $M(f,I) + \varepsilon/4$ with $\varepsilon$ from (2). Then
$$
\frac{1}{2\pi} \int_{\theta\in \Theta} \min \{ \gamma_{f_n}(\Phi(e^{i\theta})) - \gamma_{f}(\Phi(e^{i\theta})), 0 \} \, d\theta \rightarrow 0,
$$
where
$$
\Theta := \{ \theta \in[0,2\pi) : \Phi(e^{i\theta}) \in I \backslash X \}.
$$

Therefore, keeping the explicit formulas \eqref{SCformula} and \eqref{e.Phi2formula} for the change of variables in mind, it follows that for $n \geq n_0$ with $n_0$ sufficiently large, there exists a set $Y_n$ of measure bounded by $\varepsilon/4$ such that $\gamma_{f_n}(z) \geq \delta/2$ for every $z \in I \backslash (X \cup Y_n)$. Consequently, $\limsup M(f_n,I) \leq M(f,I) + \varepsilon/2$, which contradicts (2).
\end{proof}

\begin{lemma}{\label{AD05L3}}
Suppose $I \subset \partial \D$ is an arc of length $|I| \le \frac{\pi}{2}$. For $f \in C(\Omega,\D)$, $\varepsilon > 0$, $\delta > 0$, there exists $\tilde{f} \in C(\Omega,\D)$ such that $\| f - \tilde{f} \|_{\infty} < \varepsilon$, $M(\tilde{f},I) < \delta$, and $\int_{0}^1 M(\lambda \tilde{f},I) \, d\lambda < \delta$.
\end{lemma}

\begin{proof}
Given $f \in C(\Omega,\D)$, note that $\|f\|_\infty < 1$ since $\Omega$ is compact and $f$ is continuous. Place an $L^\infty$ ball $B$ of radius less than $1 - \|f\|_\infty$ around $f$ and carry out the subsequent steps entirely within this ball. This will ensure that the Lyapunov exponent exists for all sampling functions that appear.

With the dense subset $\mathfrak{L}$ of $L^\infty(\Omega,\mathbb{D})$ from \cite[Lemma~2]{AD05}, consisting of sampling functions taking finitely many values and for which the resulting process is not periodic (cf.~the assumptions of Theorem~\ref{th2}), we choose $s \in \mathfrak{L} \cap B$ such that $\|f - s\|_{\infty} < \varepsilon/2$. By Theorem \ref{th2}, we have $M(\lambda s,I) = 0$ for every $\lambda \in (0, 1)$. We choose the same construction as in \cite[Lemma~3]{AD05} to generate $f_n \in B$, for which we have, by Lemma~\ref{AD05L1}, $M(f_n,I)$, $\int_0^1 M(\lambda f_n,I) \, d\lambda \rightarrow 0$ as $n \rightarrow \infty$. By choosing $n$ large enough, we complete the proof.
\end{proof}

\begin{lemma}{\label{AD05T1}}
Suppose $I \subset \partial \D$ is an arc of length $|I| \le \frac{\pi}{2}$. Then there is a residual set of functions $f \in C(\Omega, \mathbb{D})$ such that $M(f,I) = 0$.
\end{lemma}

\begin{proof}
For $\delta > 0$, define
$$
M_{\delta}(I) = \{ f \in C(\Omega,\D) : M(f,I) < \delta \}.
$$
By Lemma~\ref{AD05L1}, $M_{\delta}(I)$ is open, and by Lemma~\ref{AD05L3}, $M_{\delta}(I)$ is dense. It follows that
$$
\{ f \in C(\Omega,\D) : M(f,I) = 0 \} = \bigcap_{n \ge 1} M_{\frac1n}(I)
$$
is residual.
\end{proof}

\begin{lemma}{\label{AD05T2}}
Suppose $I \subset \partial \D$ is an arc of length $|I| \le \frac{\pi}{2}$. Then there is a residual set of functions $f \in C(\Omega,\mathbb{D})$ such that $M(\lambda f,I) = 0$ for almost every $\lambda \in (0, 1]$.
\end{lemma}

\begin{proof}
For $\delta > 0$, define
$$
M_{\delta} = \{ f \in C(\Omega,\D) : \int_0^1 M(\lambda f, I) \, d\lambda < \delta \}.
$$
By Lemma~\ref{AD05L1}, $M_{\delta}$ is open, and by Lemma~\ref{AD05L3}, $M_{\delta}$ is dense. Thus,
$$
\bigcap_{\delta > 0}  M_{\delta}
$$
is residual. It follows that for Baire generic $f \in C(\Omega,\D)$, we have $M(\lambda f,I) = 0$ for almost every $\lambda \in (0, 1)$.
\end{proof}

\begin{theorem}\label{t.main3}
The sets
$$
\{ f \in C(\Omega,\mathbb{D}) : \Sigma_\mathrm{ac}(f) = \emptyset \}
$$
and
$$
\{ f \in C(\Omega,\mathbb{D}) : \Sigma_\mathrm{ac}(\lambda f) = \emptyset \text{ for Lebesgue almost every }
0 < \lambda \le 1 \}
$$
are residual.
\end{theorem}

\begin{proof}
Pick closed arcs $I_1, I_2,I_3, I_{4} \subset \partial \mathbb{D}$ of length $\frac{\pi}{2}$ such that
$$
\bigcup_{j=1}^4 I_{j} = \partial \mathbb{D}.
$$
By Remark~\ref{r.absenceofaccrit}, we have
\begin{equation*}
\{ f \in C(\Omega,\mathbb{D}) : \Sigma_\mathrm{ac}(f) = \emptyset \} = \bigcap_{j=1}^4 \{ f \in C(\Omega,\mathbb{D}) : M(f,I_j) = 0 \}
\end{equation*}
and
\begin{align*}
\{ & f \in C(\Omega,\mathbb{D})  : \Sigma_\mathrm{ac}(\lambda f) = \emptyset \text{ for Lebesgue a.e. }
0 < \lambda \le 1 \} \\
& = \bigcap_{j=1}^4 \{ f \in C(\Omega,\mathbb{D}) : M(\lambda f,I_j) = 0 \text{ for Lebesgue a.e. }
0 < \lambda \le 1 \}.
\end{align*}

By Lemmas~\ref{AD05T1} and \ref{AD05T2}, the sets on the right-hand side are residual. The theorem now follows since the intersection of a finite number of residual sets is also a residual set.
\end{proof}

\begin{proof}[Proof of Theorem~\ref{t.main1}]
The first statement in Theorem~\ref{t.main3} is precisely the assertion of Theorem~\ref{t.main1} and hence the latter theorem is proved.
\end{proof}

\section{Generic Zero-Measure Spectrum}\label{s.4}

In this short section we note that the map
\begin{equation}\label{e.msigmadef}
M_\Sigma: C(\Omega,\D) \to [0,\infty), \quad f \mapsto \Leb(\Sigma_f)
\end{equation}
is upper semi-continuous. The proof uses variations of ideas developed in \cite{DFL17} in the context of continuum limit-periodic Schr\"odinger operators. This semi-continuity result will then imply that $\{ f \in C(\Omega,\D) : \Leb(\Sigma_f) = 0 \}$ is a $G_\delta$ set, and hence it reduces the desired genericity statement formulated in Theorem~\ref{t.main2} to a proof that it is dense. In the case of subshifts satisfying the Boshernitzan condition, the latter property is known.

\begin{prop}\label{p.semicont}
The map $M_\Sigma$ defined in \eqref{e.msigmadef} is upper semi-continuous, that is, for every $\delta > 0$, we have that $M_\Sigma(\delta) := \{ f \in C(\Omega,\D) : \Leb(\Sigma_f) < \delta \}$ is open.
\end{prop}

Before proving Proposition~\ref{p.semicont}, we need to recall two lemmas from \cite{FOV18} and \cite{S04}. To state them, let us recall some basic notions.

For a subset $E \subseteq \partial \D$ and $\varepsilon > 0$, we denote
$$
B_\varepsilon(E) := \left\{ z \in \partial \D : \inf_{x \in E} |z - x| < \varepsilon \right\}.
$$
The \emph{Hausdorff distance} between two compact sets $F, K \subseteq \partial \D$ is defined by
$$
d_\mathrm{H}(F, K) := \inf \{ \varepsilon > 0 : F \subseteq B_\varepsilon(K) \text{ and } K \subseteq B_\varepsilon(F) \}.
$$

The following is \cite[Lemma~3.1]{FOV18}:

\begin{lemma}\label{l.linftypert}
For any pair of unitary operators $U, V$ on $\ell^2$, we have
\begin{equation}\label{e.distspec}
d_\mathrm{H}(\sigma(U), \sigma(V)) \le \|U - V\|,
\end{equation}
where $\| \cdot \|$ denotes the usual operator norm.
\end{lemma}

If $U$ and $V$ are CMV matrices, the operator norm of their difference can be estimates in terms of the $\| \cdot \|_\infty$ norm of the difference of the respective Verblunsky coefficients. Namely, the estimate (4.3.13) in \cite[Theorem~4.3.3]{S04} reads as follows:\footnote{Actually, the estimate (4.3.13) in \cite[Theorem~4.3.3]{S04} is stated for CMV matrices, while we need the estimate for extended CMV matrices. However, the proof for the case we need can be given along similar lines.}

\begin{lemma}\label{l.opnormdiff}
Consider sequences $\{ \alpha_n \}_{n \in \Z}$ and $\{ \alpha_n' \}_{n \in \Z}$ of Verblunsky coefficients and the associated extended CMV matrices $\mathcal{E}_{\alpha}$ and $\mathcal{E}_{\alpha'}$. Then,
\begin{equation}\label{e.distops}
\| \mathcal{E}_{\alpha} - \mathcal{E}_{\alpha'} \| \le 6 \sqrt{2} \|\alpha - \alpha'\|_\infty^{1/2}.
\end{equation}
\end{lemma}

\begin{proof}[Proof of Proposition~\ref{p.semicont}]
Let $\delta > 0$ be given, and let us consider $f \in M_\Sigma(\delta)$. We have to show that there exists $\varepsilon > 0$ such that every $g \in C(\Omega,\D)$ with $\|f - g\|_\infty < \varepsilon$ belongs to $M_\Sigma(\delta)$ as well.

By assumption, we have $\varepsilon' := \delta - \Leb(\Sigma_f) > 0$. By basic properties of the Lebesgue measure, we can choose finitely many open arcs $I_1, \ldots, I_m \subset \partial \D$ with
$$
\Sigma_f \subset \bigcup_{j = 1}^m I_j \quad \text{and} \quad \sum_{j = 1}^m |I_j| < \Leb(\Sigma_f) + \frac{\varepsilon'}{2}.
$$

Let us set
$$
\tilde \varepsilon := \frac{\varepsilon'}{8m} > 0 \quad \text{ and } \quad \varepsilon := \left( \frac{\tilde \varepsilon}{6 \sqrt{2}} \right)^2 > 0.
$$
By \eqref{e.distspec} and \eqref{e.distops}, if $\|f - g\|_\infty < \varepsilon$, then $\Sigma_g \subset B_{\tilde \varepsilon}(\Sigma_f)$.

Putting these two ingredients together, we obtain
$$
\Sigma_g \subset B_{\tilde \varepsilon} \left( \bigcup_{j = 1}^m I_j \right),
$$
and hence
\begin{align*}
\Leb(\Sigma_g) & \le \Leb \left( B_{\tilde \varepsilon} \left( \bigcup_{j = 1}^m I_j \right) \right) \\
& \le 4m \tilde \varepsilon + \sum_{j = 1}^m |I_j| \\
& < 4m \tilde \varepsilon + \Leb(\Sigma_f) + \frac{\varepsilon'}{2} \\
& = \delta,
\end{align*}
as desired. This completes the proof.
\end{proof}

\begin{coro}\label{c.zeromeasspecgdelta}
We have that $\{ f \in C(\Omega,\R) : \Leb(\Sigma(f)) = 0 \}$ is a $G_\delta$ set.
\end{coro}

\begin{proof}
Simply write
$$
\{ f \in C(\Omega,\R) : \Leb(\Sigma(f)) = 0 \} = \bigcap_{n \in \N} M_\Sigma \left( \tfrac1n \right)
$$
and use the fact that each $M_\Sigma(1/n)$ is open by Proposition~\ref{p.semicont}.
\end{proof}

\begin{proof}[Proof of Theorem~\ref{t.main2}]
By Corollary~\ref{c.zeromeasspecgdelta}, the set $\{ f \in C(\Omega,\R) : \Leb(\Sigma(f)) = 0 \}$ is a $G_\delta$ set, and by \cite{DL07} it is also dense under the assumptions of the theorem. Indeed, since $\Omega$ is an aperiodic subshift satisfying the Boshernitzan condition (B), \cite{DL07} shows that $\Leb(\Sigma(f)) = 0$ holds for every locally constant $f : \Omega \to \D$ (recall that $f$ is locally constant if there is an $N \in \Z_+$ such that $f(\omega)$ is determined by $\omega_{-N}, \ldots, \omega_N$). The set of locally constant $f$'s is dense in $C(\Omega,\D)$, and hence the result follows.
\end{proof}

\begin{appendix}

\section{Invariance of the Spectrum and Spectral Type}\label{s.a}

In this appendix we discuss Theorems~\ref{t.main0a} and \ref{t.main0b}. As pointed out earlier, these statements are central, essentially known to experts, but not available in the literature. For this reason we provide some comments on how they are obtained.

We remark that a related discussion can be found in \cite[Section~3]{GT94}. However, the authors of \cite{GT94} do not consider extended CMV matrices (which had not yet been introduced) and hence they do not directly discuss Theorems~\ref{t.main0a} and \ref{t.main0b}.

\subsection{A Brief Discussion of Theorem~\ref{t.main0a}}

In this subsection we briefly sketch the proof of Theorem~\ref{t.main0a}. As was mentioned in the introduction, the proof follows quickly from the proof of \cite[Theorem~10.9.11]{S05}. Let us explain why this is the case.

By symmetry it suffices to show that for every pair $\omega_1$, $\omega_2 \in \Omega$, $\sigma(\mathcal{E}_{\omega_1}) \subseteq \sigma(\mathcal{E}_{\omega_2})$. By using minimality, there is a sequence $\{n_j\}_{j \geq 1}$ such that $T^{n_j} \omega_2 \rightarrow \omega_1$ as $j \rightarrow \infty$. Due to the continuity of the sampling function $f$, $\mathcal{E}_{T^{n_j}\omega_2}$ converges strongly to $\mathcal{E}_{\omega_1}$. Thus,
$$
\sigma(\mathcal{E}_{\omega_1}) \subseteq \overline{\bigcup_{j \geq 1} \sigma(\mathcal{E}_{T^{n_j} \omega_2})} = \sigma(\mathcal{E}_{\omega_2}).
$$

As for the absolutely continuous part, since up to a finite-rank perturbation, $\mathcal{E}_{\omega}$ can be viewed as the direct sum of two half-line CMV matrices, $\mathcal{C}^{+}_{\omega}$ and $\mathcal{C}^{-}_{\omega}$, we have by the invariance of the absolutely continuous spectrum of a unitary operator under trace class perturbations \cite{C74, CP76},\footnote{The trace class perturbation theory is generally discussed primarily in the self-adjoint case, but as pointed out in the two papers mentioned here, the unitary analog follows via an application of the Cayley transform.}
\begin{equation}\label{e.acwlhl}
\sigma_{\mathrm{ac}}(\mathcal{E}_{\omega}) = \sigma_{\mathrm{ac}}(\mathcal{C}_{\omega}^{+}) \cup \sigma_{\mathrm{ac}}(\mathcal{C}_{\omega}^{-}).
\end{equation}
For the right half-line CMV matrix, Theorem~10.9.11 in \cite{S05} shows that there is a set $\Sigma_{\mathrm{ac}}^+ \subseteq \partial \D$ such that
$$
\sigma_{\mathrm{ac}}(\mathcal{C}^{+}_{\omega}) = \Sigma_{\mathrm{ac}}^+ \textrm{ for all }\omega.
$$
Similarly, one finds that there is a set $\Sigma_{\mathrm{ac}}^- \subseteq \partial \D$ such that
$$
\sigma_{\mathrm{ac}}(\mathcal{C}^{-}_{\omega})=\Sigma_{\mathrm{ac}}^- \textrm{ for all }\omega.
$$
The results and discussion in \cite[Section~10.9]{S05} also imply that $\Sigma_{\mathrm{ac}}^{+} = \Sigma_{\mathrm{ac}}^{-} =: \Sigma_{\mathrm{ac}}(f)$. Therefore, $\sigma_{\mathrm{ac}}(\mathcal{E}_{\omega}) = \Sigma_{\mathrm{ac}}(f)$ for all $\omega$.
\qed

\subsection{A Brief Discussion of Theorem~\ref{t.main0b}}

In this subsection we briefly sketch the proof of Theorem~\ref{t.main0b}. As was mentioned in the introduction, the statement in Theorem~\ref{t.main0b} about the spectrum being almost everywhere constant is \cite[Theorem~10.16.1]{S05}. To get the other statements in Theorem~\ref{t.main0b} one needs to run an analogous proof with the spectral projections replaced by the respective partial spectral projections. The key technical point to address here concerns measurability. Once that has been addressed, the result follows readily from covariance.

Let us first state an elementary and well-known lemma (see, e.g., \cite[Lemma~4.1.2]{DF} for the version of this lemma for finite Borel measures on the real line).

\begin{lemma}{\label{S-part}}
Fix a finite Borel measure $\mu$ on $\partial \D$ and a countable dense subset $S$ of $\partial \D$. Let $\mathcal{J}$ denote the countable collection of all finite unions of open arcs in $\partial \D$ whose endpoints belong to $S$. Moreover, let $\mathcal{J}_{\varepsilon}$ denote the collection of all $J \in \mathcal{J}$ such that $\mathrm{Leb}(J) < \varepsilon$. Then
$$
\mu_s(B) = \lim_{\varepsilon \rightarrow 0} \sup_{J \in \mathcal{J}_{\varepsilon}} \mu(B\cap J)
$$
for any Borel set $B$.
\end{lemma}

To complete the proof of Theorem~\ref{t.main0b}, it suffices to prove the family $\{ \mathcal{P}_{\omega}^{\bullet}(I) \}_{\omega \in \Omega}$ is weakly measurable (i.e., $\omega \mapsto \langle \phi, \mathcal{P}_{\omega}^{(\bullet)}(I) \psi \rangle$ defines a measurable function from $\Omega$ to $\mathbb{C}$ for all $\phi, \psi \in \mathcal{H}$) for $\bullet \in \{ \mathrm{ac,sc,pp} \}$ and for $I \subset \partial\D$. Here $\mathcal{P}_{\omega}(I)$ is the spectral projection onto the arc $I$ associated with $\mathcal{E}_{\omega}$. $\mathcal{P}_{\omega}^{(\mathrm{ac})}(I)$, $\mathcal{P}_{\omega}^{(\mathrm{sc})}(I)$ and $\mathcal{P}_{\omega}^{(\mathrm{pp})}(I)$ are given by $\mathcal{P}_{\omega}^{(\bullet)}(I) = \mathcal{P}_{\omega}(I) \mathcal{P}_{\omega}^{(\bullet)}$, $\bullet \in \{ \mathrm{ac,sc,pp} \}$, respectively.

By the RAGE theorem for unitary operators (see, e.g., \cite[Theorem A.2]{FO17}),
\begin{equation}\label{eqp}
\langle \phi, \mathcal{P}_{\omega}^{(\mathrm{c})} \psi \rangle = \lim_{J \rightarrow \infty} \lim_{N \rightarrow \infty} \frac{1}{2N+1} \langle \phi, \mathcal{E}_{\omega}^{-N} (I - P_{J}) \mathcal{E}_{\omega}^N \psi \rangle,
\end{equation}
where $P_J$ denotes the orthogonal projection onto the linear span of $\delta_{-J},\ldots,\delta_J$ in $\ell^2(\Z)$ and $\phi,\psi\in\ell^2(\Z)$. Since $\mathcal{E}_{\omega}$ is weakly measurable, every polynomial in $\mathcal{E}_{\omega}$ is also weakly measurable. Then $\mathcal{P}_{\omega}(I)$ and $\mathcal{P}_{\omega}^{(\mathrm{c})}$ are weakly measurable. Thus, $\mathcal{P}_{\omega}^{(\mathrm{c})}(I)$ is weakly measurable since $\mathcal{P}_{\omega}^{(\mathrm{c})}(I) = \mathcal{P}_{\omega}(I) \mathcal{P}_{\omega}^{(\mathrm{c})}$.

Due to Lemma~\ref{S-part}, we have
$$
\langle \phi, \mathcal{P}_{\omega}^{(\mathrm{s})}(I) \psi \rangle = \lim_{\varepsilon \rightarrow 0} \sup_{J \in \mathcal{J}_{\varepsilon}} \langle \phi, \mathcal{P}_{\omega} (I \cap J) \psi \rangle.
$$
By weak measurability of $\mathcal{P}_{\omega}$, polarization, and countability of $\mathcal{J}_{\varepsilon}$, it follows that $\mathcal{P}_{\omega}^{(\mathrm{s})}(I)$ is weakly measurable.

Therefore,
$$
\mathcal{P}_{\omega}^{(\mathrm{ac})} = \mathcal{P}_{\omega}^{(\mathrm{c})} (I - \mathcal{P}_{\omega}^{(\mathrm{s})}), \quad
\mathcal{P}_{\omega}^{(\mathrm{sc})} = \mathcal{P}_{\omega}^{(\mathrm{c})} \mathcal{P}_{\omega}^{(\mathrm{s})}, \quad
\mathcal{P}_{\omega}^{(\mathrm{pp})} = \mathcal{P}_{\omega}^{(\mathrm{s})} (I - \mathcal{P}_{\omega}^{(\mathrm{c})})
$$
are all weakly measurable. \qed

\section{Kotani Theory for Extended CMV Matrices}\label{a.appb}

In this section we state some of the key results from Kotani theory that we need in the main part of the paper. Since the statements we need are not discussed in the exact same form in the literature, we provide some explanation regarding their proofs. Specifically, Kotani theory for OPUC and CMV matrices has been discussed in \cite{GT94} and \cite{S05}. However, \cite{GT94} was written prior to the general use of CMV matrices, and hence does not discuss results for these matrices, and \cite{S05} does on the one hand focus on CMV rather than extended CMV matrices, and on the other hand, and much more importantly, the statements of the main theorems of Kotani theory in \cite{S05} assume the condition \eqref{eq5}, which we cannot assume if we pursue the results the present paper wants to establish. Thus, our main goal in writing this appendix is to explain why the main results from Kotani theory for extended CMV matrices, once suitably formulated and interpreted, hold without the assumption \eqref{eq5}.

The Thouless formula connects the Lyapunov exponent and the \emph{density of zeros measure}. Establishing the Thouless formula is an important preliminary step to developing Kotani theory. The proof of the Thouless formula given in \cite{S05} assumes that the density of zeros measure exists as a measure on $\partial \D$; see especially \cite[Theorems~10.5.8 {\&} 10.5.26]{S05}.

However, if in this way one starts the discussion of Kotani theory from the density of zeros measure, one is faced with the unpleasant realization that in the trivial case ($f \equiv 0$), all zeros sit at $0$, and hence the density of zeros measure is the Dirac mass at $0$ in this case, and it is in particular not a measure on $\partial \D$. An intriguing conjecture, \cite[Conjecture~10.5.23]{S05}, states that in all other cases, the density of zeros measure does exist as a measure on $\partial \D$. This conjecture is still (wide) open. On the other hand, \cite[Theorem~10.5.19]{S05} shows that under the assumption \eqref{eq5}, the density of zeros measure does exist as a measure on $\partial \D$.

Since we do not want to assume \eqref{eq5}, given the applications presented in this paper, it is important to note that an alternative version of the Thouless formula connects the Lyapunov exponent and the \emph{density of states measure}. The latter exists without any assumption on $f$. This change of perspective and the generality in which the formula holds are discussed in the second remark following \cite[Theorem~10.5.26]{S05}. We state the corresponding result in Theorem~\ref{thouless} below.

Let us first specify the setting: $(\Omega, d\beta)$ is a probability measure space, $T : \Omega \rightarrow \Omega$ is an invertible measure-preserving ergodic transformation, and $f : \Omega \rightarrow \mathbb{D}$ a measurable function. The (in the language of Simon \cite{S05}) \textit{stochastic Verblunsky coefficients} associated to $(\Omega, d\beta, T, f)$ are the random variables on $\Omega$ given by
$$
\alpha_n(\omega) = f(T^n \omega) , \; n \in \mathbb{Z}.
$$

The extended CMV matrix associated with the sequence $\{ \alpha_n(\omega) \}_{n \in \Z}$ is denoted by $\mathcal{E}_\omega$. Averaging the spectral measure corresponding to the pair $(\mathcal{E}_\omega , \delta_0)$ with respect to $d\beta$, we obtain the \emph{density of states measure} $d\nu$:
$$
\int g(\theta) \, d\nu(\theta) = \int \langle \delta_0, g(\mathcal{E}_{\omega}) \delta_0 \rangle \, d\beta(\omega)
$$
for any $g\in C(\partial \D)$.

\begin{theorem}\label{thouless}
Assume \eqref{eq4}. Then there is $\rho_{\infty} \in (0,1]$ such that for almost every $\omega \in \Omega$ with respect to $d \beta$, we have
$$
\rho_{\infty} = \lim_{N \rightarrow \infty} \left( \prod_{n = 0}^{N-1}\rho_n(\omega) \right)^{\frac{1}{N}}.
$$
Moreover we have
$$
\gamma_{f}(z) = -\log {\rho_{\infty}} - \int \log{|z-y|^{-1}} \, d\nu(y).
$$
\end{theorem}

\begin{proof}
As indicated above, this theorem is a consequence of (the proof of) \cite[Theorem~10.5.26]{S05} and the discussion of this result in \cite{S05}, especially the second remark after the statement of \cite[Theorem~10.5.26]{S05}.
\end{proof}

The two key Kotani theory statements for extended CMV matrices we need in this paper are the following. The first is the identity \eqref{e.kotanithm}.

\begin{theorem}\label{th1}
Assume \eqref{eq4}. Then, $\Sigma_\mathrm{ac}(f) = \overline{\{ z \in \partial \D : \gamma_f(z) = 0 \}}^{\mathrm{ess}}$.
\end{theorem}

\begin{proof}
This is essentially \cite[Theorem~10.11.1]{S05}. However, there are two points we need to address here. First, we have switched from the density of zeros measure to the density of states measure in the discussion of the Thouless formula, while Simon proves \cite[Theorem~10.11.1]{S05} based on the Thouless formula that involves the density of zeros measure. Second, that theorem makes a statement about the almost sure absolutely continuous spectrum of standard (rather than extended) CMV matrices, while our statement involves extended CMV matrices.

\medskip

To address the first item, we need to inspect the proof \cite[Theorem~10.11.1]{S05} and note that the feature of the Thouless formula which is used there is the fact that the right-most term is the (negative of) the logarithmic energy of a probability measure on the unit circle. At that point it no longer matters whether this measure is obtained as a limit of the distribution of zeros of orthogonal polynomials or the $\beta$-average of the spectral measure associated with the pair $(\mathcal{E}_\omega , \delta_0)$.

\medskip

To address the second item, recall that \eqref{e.acwlhl} expresses the absolutely continuous spectrum of the extended (i.e., whole-line) CMV matrix in terms of the absolutely continuous spectra of the two half-line restrictions:
$$
\sigma_{\mathrm{ac}}(\mathcal{E}_{\omega}) = \sigma_{\mathrm{ac}}(\mathcal{C}_{\omega}^{+}) \cup \sigma_{\mathrm{ac}}(\mathcal{C}_{\omega}^{-})
$$
and as was explained in the proof of Theorem~\ref{t.main0a} above,\footnote{To be a bit more precise, the discussion in the proof of Theorem~\ref{t.main0a} is given in the topological setting considered there (where $T$ is a homeomorphism), but the arguments hold equally well in the measurable setting we consider here.} for $\beta$-almost every $\omega \in \Omega$, we have
$$
\Sigma_\mathrm{ac}(f) = \sigma_{\mathrm{ac}}(\mathcal{E}_{\omega}) = \sigma_{\mathrm{ac}}(\mathcal{C}_{\omega}^{+}) = \sigma_{\mathrm{ac}}(\mathcal{C}_{\omega}^{-}).
$$
Now, \cite[Theorem~10.11.1]{S05} states that for $\beta$-almost every $\omega \in \Omega$,
$$
\sigma_{\mathrm{ac}}(\mathcal{C}_{\omega}^{+}) = \overline{\{ z \in \partial \D : \gamma_f(z) = 0 \}}^{\mathrm{ess}}.
$$
Combining these two identities, the theorem follows.
\end{proof}

The second is the vanishing of these sets in the aperiodic finitely valued setting, used in the proof of Lemma~\ref{AD05L3}.

\begin{theorem}\label{th2}
Assume \eqref{eq4}. If the sampling function $f$ takes only finitely many values and the process is not periodic {\rm (}i.e., it is false that for some fixed $p$ and almost every $\omega$, $\alpha_{n+p}(\omega) = \alpha_n(\omega)$ for all $n \in \mathbb{Z}${\rm )}, then
$$
\Leb( \{ z \in \partial \D : \gamma_f(z) = 0 \} ) = 0.
$$
In particular, in this case we have $\Sigma_\mathrm{ac}(f) = \emptyset$.
\end{theorem}

\begin{proof}
The first statement is \cite[Theorem~10.11.3]{S05}. The second statement then follows from Theorem~\ref{th1}.
\end{proof}

\end{appendix}
{}

\end{document}